\documentclass[12pt]{amsart}
\usepackage{amsfonts,amsmath,amssymb,amscd,amsthm,enumerate}
\usepackage{amsrefs}
\usepackage{geometry}
\usepackage{adjustbox}
\usepackage[all]{xy}
\usepackage{stmaryrd}
\newtheorem{theorem}{Theorem}
\newtheorem{prop}[theorem]{Proposition}
\newtheorem{lem}[theorem]{Lemma}

\theoremstyle{definition}

\newtheorem{rem}[theorem]{Remark}
\newtheorem{mydef}[theorem]{Definition}
\newtheorem{example}[theorem]{Example}

\renewcommand{\epsilon}{\varepsilon}

\def\<{\langle}
\def\>{\rangle}

\newcommand{\BigWedge}{\mathord{\adjustbox{valign=B,totalheight=.6\baselineskip}{$\bigwedge$}}}

\usepackage{multicol} 

\begin{document}
\title{2-step nilpotent $L_\infty$-algebras and Hypergraphs}
\author{Marco Aldi and Samuel Bevins}

\begin{abstract}
We describe a procedure to attach a nilpotent strong homotopy Lie algebra to every simple hypergraph and prove that two hypergraphs are isomorphic if and only if the corresponding strong homotopy Lie algebras are isomorphic. As an application, we characterize hypergraphs admitting a system of distinct representatives in terms of symplectic forms on the corresponding strong homotopy Lie algebra. We conclude with a combinatorial description of the cohomology of these strong homotopy Lie algebras in low degree.
\end{abstract}
\maketitle

\section{Introduction}

Motivated by the study of existence of Anosov diffeomorphisms on nilmanifolds, Dani and Mainkar \cite{DaniMainkar05} introduced a construction that to each simple graph $G$ associates a 2-step nilpotent Lie algebra $\mathcal L(G)$. An appealing feature of the resulting 2-step nilpotent Lie algebras is that much of their structure (and that of their associated Lie groups and nilmanifolds) can be described in purely graph-theoretical terms. For this reason, the Dani-Mainkar construction has been exploited to exhibit interesting examples of symplectic nilmanifolds \cite{PouseeleTirao09}, of Einstein solvmanifolds \cite{LauretWill11}, and of geodesic flows on nilmanifolds\cite{Ovando20}. 

In a different direction, Mainkar \cite{Mainkar15} has shown that two simple graphs $G_1$ and $G_2$ are isomorphic if and only if $\mathcal L(G_1)$ and $\mathcal L(G_2)$ are isomorphic as Lie algebras. This result suggests the possibility to translate graph theoretic notions into the language of Lie algebras. Conversely, looking at $\mathcal L(G)$ through the lenses of Lie theory can shed new light on the underlying graph $G$. For instance, it is natural to look at the Cartan-Chevalley-Eilenberg cohomology of $\mathcal L(G)$ as an invariant of $G$. While the general structure of the cohomology of $\mathcal L(G)$ is not well understood, it is known \cite{PouseeleTirao09} that $H^2(\mathcal L(G))$ detects triangles in $G$. 

In this paper we extend the Dani-Mainkar construction to simple hypergraphs $G$ in such a way that the corresponding Lie-theoretic object $\mathcal L(G)$ is a 2-step nilpotent $L_\infty$-algebra. Our first main result is that Mainkar's theorem still holds: given two simple hypergraphs $G_1$ and $G_2$, $G_1$ is isomorphic to $G_2$ if and only if $\mathcal L(G_1)$ is isomorphic to $\mathcal L(G_2)$ as $L_\infty$-algebras. As a result, simple hypergraphs can be thought of as particular types of $L_\infty$-algebras.

Our second main result is about the existence of (graded) symplectic forms on $L_\infty$-algebras. We show that $\mathcal L(G)$ admits a symplectic structure if and only if $G$ has a system of distinct representatives. This connection is intriguing in light of the role played by systems of distinct representatives in the study of the Quantum Satisfiability Problem \cites{LMSS10,AdBGS21}.

To further illustrate our construction, we conclude our paper with a study of the cohomology of the Maurer-Cartan algebra of $\mathcal L(G)$. One novelty is that while in the case of graphs the Cartan-Chevalley-Eilenberg algebra of the corresponding Dani-Mainkar Lie algebra is always finite dimensional, this is not necessarily the case for simple hypergraphs with edges containing an odd number of vertices. Nevertheless, we prove that $H^i(\mathcal L(G))$ is finite dimensional for every integer $i$ and for every simple finite hypergraph $G$. Hence the dimensions of these spaces can be thought of as invariants of the hypergraph $G$ which, at least in simple cases, we are able to explicitly express in terms of combinatorial data on $G$.

\section{The 2-step nilpotent $L_\infty$-algebra of a hypergraph}

We find it convenient to use grading conventions for $L_\infty$-algebras adapted from \cite{Getzler09}, the original source for nilpotent $L_\infty$-algebras.

\begin{mydef}\label{def:1}
Let $\mathfrak g=\bigoplus_i \mathfrak g^i$ be a $\mathbb Z$-graded real vector space that is finite dimensional in each degree and bounded below. Let $\mathfrak g[1]^\vee$ be the shifted dual, so that $(\mathfrak g[1]^\vee)^i$ is the dual vector space of $\mathfrak g^{1-i}$. For each $k>0$, let $l_k$ be a sequence of graded antisymmetric operations on $\mathfrak g$ of degree $2-k$  (i.e.\ $l_k:\BigWedge^k \mathfrak g\to \mathfrak g$ is homogeneous linear of degree 2) and let $d_k:\mathfrak g[1]^\vee\to {\rm Sym}(\mathfrak g[1]^\vee)$ be the corresponding (degree 1) transpose maps. We say that the maps $\{l_k\}$ form an {\it $L_\infty$-algebra (or strong homotopy Lie algebra)} on $\mathfrak g$ if the unique extension of $d=\sum_{k=1}^\infty d_k$ to ${\rm Sym}(\mathfrak g[1]^\vee)$ by graded derivation is a differential i.e.\ $d^2=0$. If this is the case, we denote by $C^\bullet(\mathfrak g)=({\rm Sym}(\mathfrak g[1]^\vee),d)$ its {\it Maurer-Cartan algebra} (an explanation for this terminology may be found in \cite{Huebschmann17})   and by $H^\bullet(\mathfrak g)$ the corresponding cohomology.
\end{mydef}

\begin{rem} In \cite{Huebschmann10} and references therein, an $L_\infty$-structure on
a graded vector space $\mathfrak g$
is defined as a coalgebra differential on the graded
symmetric coalgebra $S'[s\mathfrak g]$ on the suspension $s\mathfrak g$ of $\mathfrak g$. Under the circumstances of the present paper, this characterization is equivalent to Definition \ref{def:1} above.
\end{rem}

\begin{mydef}
Two $L_\infty$-algebras with underlying graded vector spaces $\mathfrak g_1$, $\mathfrak g_2$ are {\it isomorphic} if there exists an isomorphism $f:\mathfrak g_1\to \mathfrak g_2$ of graded vector spaces such that the induced map $f^\bullet:C^\bullet(\mathfrak g_2)\to C^\bullet(\mathfrak g_1)$ of differential graded algebras is an isomorphism.
\end{mydef}

\begin{rem}
Equivalently, $f:\mathfrak g_1\to \mathfrak g_2$ induces an isomorphism of $L_\infty$-algebras if $l_{k,2} (f(x_1),\ldots,f(x_k))=f(l_{k,1}(x_1,\ldots,x_k))$ for each $k\ge 1$ and  $x_1,\ldots,x_k\in \mathfrak g_1$, where $l_{k,1}$ (resp.\ $l_{k,2}$) are the $L_\infty$ operations on $\mathfrak g_1$ (resp.\ on $\mathfrak g_2$).
\end{rem}

\begin{mydef}
The {\it lower central filtration} \cite{Getzler09} of an $L_\infty$-algebra $\{l_k\}$ on $\mathfrak g$ is the canonical decreasing filtration defined inductively by $F^1\mathfrak g=\mathfrak g$ and, for each $i>1$,
\begin{equation}
    F^i\mathfrak g = \sum_{i_1+\cdots +i_k=i} l_k(F^{i_1}\mathfrak g,\ldots,F^{i_k}\mathfrak g)\,.
\end{equation}
An $L_\infty$-algebra $\mathfrak g$ is {\it nilpotent} if $F^N\mathfrak g=0$ for all sufficiently large $N$. 
\end{mydef}

\begin{mydef}
We say that an $L_\infty$-algebra $\{l_k\}$ on $\mathfrak g$ is {\it 2-step nilpotent} if 
\begin{equation}\label{eq:2}
    l_k(l_j(\mathfrak g,\ldots, \mathfrak g),\mathfrak g,\ldots,\mathfrak g)=0
\end{equation}
for all $k,j\ge 1$. 
\end{mydef}

\begin{rem}
Let $\{l_k\}$ be a 2-step nilpotent $L_\infty$-algebra on $\mathfrak g$ such that $l_k=0$ for all $k\ge N$. Then $F^N\mathfrak g=0$, and thus $(\mathfrak g,\{l_k\})$, is nilpotent. 
\end{rem}

\begin{mydef}
We define the {\it commutator} of an $L_\infty$-algebra $\{l_k\}$ on $\mathfrak g$ to be the $L_\infty$-subalgebra 
\begin{equation}
    l(\mathfrak g)=\sum_{k=1}^n l_k(\mathfrak g,\ldots,\mathfrak g)\,.
\end{equation}
\end{mydef}

\begin{example}
If $\mathfrak g=\mathfrak g^0$ and $l_k=0$ for all $k\neq 2$, then $(\mathfrak g, l_2)$ is an $L_\infty$-algebra if and only if it is a Lie algebra. In this case, $l(\mathfrak g)=l_2(\mathfrak g,\mathfrak g)=[\mathfrak g,\mathfrak g]$ and $C^\bullet(\mathfrak g)$ is the ordinary Cartan-Chevalley-Eilenberg algebra of the Lie algebra $\mathfrak g$. Moreover, $\mathfrak g$ is 2-step nilpotent as an $L_\infty$-algebra if and only if it is 2-step nilpotent as a Lie algebra. 
\end{example}

\begin{mydef}
A {\it hypergraph} $G$ is a set (the {\it vertices}) $V(G)$ together with, for each $k>0$, collections (the $k$-{\it edges}) of (unordered) $k$-element subsets of $V(G)$. We use the notation $E(G)=\bigcup_k E_k(G)$ for the set of all {\it edges} of $G$. We say that $G$ is {\it finite} if $V(G)$ is a finite set. We say that $G$ is ${\it simple}$ if the symmetric difference of any two distinct edges is nonempty. A hypergraph $G$ is {\it $k$-uniform} if $E_i(G)= \emptyset$ for all $i\neq k$.
\end{mydef}

\begin{mydef}\label{def:6}
Let $G$ be a finite simple hypergraph with $n$ vertices and let $V$ be the real vector space generated the elements $x_i$ labeled by $i\in V(G)$. For each $k>0$ let $W_k$ be the subspace of $\BigWedge^k V$ generated by $x_{i_1}\wedge \cdots \wedge x_{i_k}$ such that $\{i_1,\ldots,i_k\}\notin E_k(G)$. {\it The 2-step nilpotent $L_\infty$-algebra of $G$} has underlying graded vector space \begin{equation}
    \mathcal L(G) = V[0] \oplus (V/W_1)[1]\oplus((\BigWedge^2 V)/W_2)[0]\oplus\cdots  \oplus  ((\BigWedge^n V)/W_n)[2-n]\,.
\end{equation}
The $L_\infty$-algebra operations $l_k$ are given by the canonical projections $\BigWedge^k V\to (\BigWedge^k V)/W_k$ i.e.\
\begin{equation}
    l_k(x_{i_1},\ldots,x_{i_k})=x_{i_1}\wedge \cdots \wedge x_{i_k} \mod W_k
\end{equation}
for all $x_{i_1},\ldots,x_{i_k}\in V$. In all other cases (when at least one of the arguments is in $(\BigWedge^p V)/W_i[2-p]$ for some $p\in\{1,\ldots,n\}$), $l_k$ is declared to vanish.
\end{mydef}

\begin{rem}
If $G$ is a finite simple graph, then $\mathcal L(G)$ coincides with the 2-step nilpotent Lie algebra constructed by Dani and Mainkar \cite{DaniMainkar05}.
\end{rem}

\begin{rem}
In practice, it is convenient to fix an order on $V(G)$, or equivalently, to identify $V(G)$ with $\{1,\ldots,|V(G)|\}$ with the standard order. Then $\BigWedge^k V/W_k$ acquires a canonical basis consisting of elements of the form $x_{i_1}\wedge \cdots \wedge x_{i_k}$ where $i_1<i_2<\cdots<i_k$ and $\{i_1,\ldots,i_k\}\in E_k(G)$. Let $x_i^*$, $x_{\{i_1,\ldots,i_k\}}^*$ be the elements of the dual basis labeled by all vertices $i\in V(G)$ and edges $\{i_1,\ldots, i_k\}\in E(G)$. Then the components $d_k$ of the differential of the Maurer-Cartan algebra $C^\bullet(\mathcal L(G))$ can be written explicitly as $d_k(x_{\{i_1,\ldots,i_k\}}^*)=x_{i_1}^*\cdots x_{i_k}^*$.  
\end{rem}

\begin{theorem}
Two finite simple hypergraphs $G$ and $G'$ are isomorphic if and only if $\mathcal L(G)$ and $\mathcal L(G')$ are isomorphic as $L_\infty$-algebras.
\end{theorem}

\begin{proof} The particular case in which $G$ and $G'$ are graphs is proved in \cite{Mainkar15}. The general case is proved along the same lines, we provide a sketch for sake of completeness. If $f:G\to G'$ is an isomorphism of hypergraphs, it is easy to see that the induced map $f:\mathcal L(G)\to \mathcal L(G')$ is an isomorphism of $L_\infty$-algebras. For the converse, assume that $f:\mathcal L(G)\to \mathcal L(G')$ is an isomorphism of $L_\infty$-algebras and that $E_k(G)$ and $E_k(G')$ are empty for all $k$ greater than a fixed $n\in \mathbb N$.  Let $V$ be as in Definition \ref{def:6} i.e.\ the vector space generated by vectors $x_i$ labeled by $i\in V(G)$. Let $T$ be the group of all $L_\infty$-algebra automorphisms $\tau$ of $\mathcal L(G)$ such that $\tau_{|_V}\in {\rm GL}(V)$. Let $\mathcal G\subseteq {\rm GL}(V)$ be the subgroup of all linear automorphisms of the form form $\tau_{|_V}$ for some $\tau\in T$. Equivalently, $\mathcal G$ is the group of automorphisms $g$ of $V$ such that the induced maps $\wedge^k g:\BigWedge^k V\to \BigWedge^k V$ preserve $W_k$ for all $k=1,\ldots,n$. Since these conditions can be expressed as polynomial equations, $\mathcal G$ is an algebraic group. Let $D$ be the subgroup of ${\rm GL}(V)$ consisting of all automorphisms diagonalized by the standard basis of vertices of $G$. Since $D(W_k)\subseteq W_k$ for all $k=1,\ldots,n$, then $D$ is a subgroup of $\mathcal G$.

We construct a new hypergraph $G''$ in the following way. Let $V'$ be the vector space generated by $V(G')$ and let $W_k'$ be such that
\begin{equation}
\mathcal L(G') = V'[0] \oplus (V'/W'_1)[1]\oplus((\BigWedge^2 V')/W_2')[0]\oplus\cdots  \oplus  ((\BigWedge^n V')/W_n')[2-n]\,,
\end{equation}
where $n = |V(G')|$.
Let $\pi:\mathcal L(G')\to V'$ be the canonical projection. We define $G''$ to be the hypergraph with vertices 
\begin{equation}
    V(G'')=\{\pi(f(x_i))\,|\,i\in V(G)\}
\end{equation}
and edges
\begin{equation}
    E_k(G'')=\{\{\pi(f(x_{i_1})),\ldots,\pi(f(x_{i_k}))\}\,|\,\{i_1,\ldots,i_k\}\in E_k(G)\}\,.
\end{equation}
Let $l(\mathcal L(G'))$ be the commutator of $\mathcal L(G')$. Since $f$ is an isomorphism of $L_\infty$-algebras, we have $\sum_{i \in V(G)}a_i\pi(f(x_i)) = 0$. This implies that $f\left(\sum_{i \in V(G)}a_i x_i \right)$ is an element of $l(\mathcal L(G'))$ and thus $\sum_{i \in V(G)}a_i x_i$ is an element of $l(\mathcal L(G))$. On the other hand, by construction $x_i\in V$ for every $i\in V(G)$. Therefore, $\sum_{i \in V(G)}a_ix_i = 0$. Since $\{x_1,\ldots,x_{n}\}$ is a linearly independent set, then  $a_i = 0$ for all $i\in V(G)$. Since $\mathcal L(G)$ and $\mathcal L(G')$ are isomorphic, then
\begin{equation}
    |V(G)|=\dim (\mathcal L(G)/l(\mathcal L(G)))=\dim (\mathcal L(G')/l(\mathcal L(G')))=|V(G')|\,. 
\end{equation}
On the other hand, $|V(G')|=|V(G'')|$ because $V(G')$ and $V(G'')$ are both bases of the same vector space. Hence $\pi\circ f:V(G)\to V(G'')$ is a bijection identifying edges of $G$ with edges of $G''$ and thus defining a graph isomorphism.

We claim that the subgroup $D''\subseteq {\rm GL}(V')$ consisting of elements diagonalized by the basis $V(G'') = \{\pi(f(x_i))\, |\, i \in V(G)\}$ is a subgroup of $\mathcal G'$. Let $d'' \in D''$ with eigenvalues $d''_i$ such that, $d''(\pi(f(x_i))) = d''_i\pi(f(x_i))$ for some nonzero $d''_i \in \mathbb R$ and for all $i \in V(G)$. Take $i_1',\ldots, i_k' \in V(G')$ where $\{i_1',\ldots , i_k'\} \notin E_k(G')$. Since  $x_{i_j'}\in V'$ and $V(G'')$ is a basis of $V'$, we can represent these elements as 
\begin{equation}
    x_{i_j'} = \displaystyle{\sum_{p \in V(G)}a_{p,j}\pi(f(x_p))}
    \end{equation}
for suitable coefficients $a_{p,j}\in \mathbb R$. We observe that the subspace $W_k'$ is invariant under the action of $d''$ in the following way. Since $l_k(x_{i_1'},\ldots, x_{i_k'}) = 0 \mod W_k'$, we see that
\begin{equation}
    l_k(x_{i_1'},\ldots, x_{i_k'}) = l_k\left( \sum_{p\in V(G)}a_{p,1}\pi(f(x_p)), \ldots , \sum_{p\in V(G)}a_{p,k}\pi(f(x_p)) \right) = 0 \mod W_k'\,.
\end{equation}
Therefore, since $l_k(v_1'+w_1',\ldots, v_k'+w_k') = 0$ if and only if $l_k(v_1',\ldots, v_k') = 0$ for all $v_j' \in V'$ and $w_j' \in \bigoplus_{i= 1}^n\left(\BigWedge^iV'\right)/W_i'$,
\begin{equation}    
    l_k\left( \sum_{p\in V(G)}a_{p,1}x_p, \ldots , \sum_{p\in V(G)}a_{p,k}x_p  \right) = 0\mod W_k\,.
\end{equation}
Let $\sigma \in {\rm GL}(V)$ such that $\sigma(x_i) = d_i''x_i$ for each $i \in V(G)$. Then $\sigma$ is in $D$, the collection of automorphisms of $V$ diagonalized by the $\{x_i\}_{i\in V(G)}$. Since $D$ is a subgroup of $\mathcal G$, we have 
\begin{equation}
 l_k\left( \sum_{p\in V(G)}a_{p,1}\sigma(x_p),\ldots ,\sum_{p \in V(G)}a_{p,k}\sigma(x_p)\right) = 0 \mod W_k\,,
\end{equation}
which implies 
\begin{equation}
    l_k\left( \sum_{p \in V(G)}a_{p,1} d_{i_1}''x_p,\ldots, \sum_{p\in V(G)}a_{p,k} d_{i_k}''x_p \right)=0 \mod W_k \,,
\end{equation}
and thus 
\begin{equation}
l_k \left( \sum_{p\in V(G)}a_{p,1} d_{i_1}''\pi(f(x_p)),\ldots , \sum_{p \in V(G)}a_{p,k} d_{i_k}''\pi(f(x_p)) \right) = 0 \mod W_k'\,.
\end{equation}
Hence $l_k(d''(x_{i_1}),\ldots , d''(x_{i_k})) =0$ modulo $W_k'$ and $l_k(d''(x_{i_1'}),\ldots, d''(x_{i_k'})) = 0$ in $\mathcal L(G')$. Hence $D''$ is a subgroup of $\mathcal G'$. Upon passing to the complexification of all vector spaces involved, $D'$ and $D''$ are maximal tori in the connected component of the identity of the algebraic group $\mathcal G'$. Therefore (see e.g.\ \cite{LAG09}) $D'$ and $D''$ are conjugated i.e.\  there exists an element $g \in \mathcal G'$ such that 
\begin{equation}
D'=gD''g^{-1}\,.
\end{equation}
Then for each $i' \in V(G')$, choose nonzero scalars $d_{i'}'$ such that (unless $\{i_1',\ldots, i_m'\}=\{j_1',\ldots, j_m'\}$), 
\begin{equation}\label{eq:17}
{\prod_{l=1}^m d_{i_l}' \neq \prod_{l=1}^m d_{j_l}'}\,.
\end{equation} Then, it follows that for each $d'\in D'$, there exists an element $d'' \in D''$ such that $d' = gd''g^{-1}$. Then $d'$ and $d''$ are similar and share the same eigenvalues and so are the same up to a permutation. Therefore, a bijection $f'$ exists between $V(G')$ and $V(G'')$ such that
\begin{equation}
d''(f'(x_{i'})) = d_i'f'(x_{i'}), \quad \forall i' \in V(G').
\end{equation}
Lastly, we establish that $\{i_1',\ldots, i_k'\} \in E_k(G')$ if and only if $\{f'(i_1'),\ldots, f'(i_k')\} \in E_k(G'')$.
We know that $\{i_1',\ldots, i_k'\} \in E_k(G')$ if and only if $l_k(x_{i_1'},\ldots, x_{i_k'}) \neq 0$ and $\{f'(i_1'),\ldots , f'(i_k')\} \in E_k(G'')$ if and only if $l_k(f'(x_{i_1'}),\ldots , f'(x_{i_k'}))\neq 0$. Then $f'(x_{i'})= \pi(f(x_i))$ for some $i \in V(G)$, so $\{f'(i_1'),\ldots, f'(i_k')\} \in E_k(G'')$ if and only if $\{i_1,\ldots, i_k\} \in E_k(G)$. Since $f$ is an $L_\infty$-algebra isomorphism and $\pi$ is the canonical linear map, it follows that $\{i_1,\ldots, i_k\} \in E_k(G)$ if and only if $l_k(\pi(f(x_{i_1})),\ldots , \pi(f(x_{i_k}))) \neq 0$ in $\mathcal L(G')$. If follows from \eqref{eq:17} that $\prod_{j=1}^kd_{i_j'}'$ is an eigenvalue of an extended automorphism $d'$ of $\mathcal L(G')$ if and only if $l_k(f'(x_{i_1'}),\ldots , f'(x_{i_k'})) \neq 0$ in $\mathcal L(G'')$. Since $d' = gd''g^{-1}$ for some $g \in \mathcal G'$,  
\begin{equation}
    l_k(g^{-1}(d'(x_{i_1'})),\ldots, g^{-1}(d'(x_{i_k'})))  = \left( \prod_{j=1}^kd_{i_j'}'\right) l_k(g^{-1}(x_{i_1'}),\ldots, g^{-1}(x_{i_k'})).
\end{equation}
Moreover, $g$ being an $L_\infty$-automorphism, $l_k(g^{-1}(x_{i_1'}),\ldots, g^{-1}(x_{i_k'}))\neq 0$ if and only if $l_k(x_{i_1'},\ldots, x_{i_k'}) \neq 0$. Thus, $\prod_{j=1}^kd_{i_j'}'$ is an eigenvalue of $d''$ if and only if $ \prod_{j=1}^kd_{i_j'}'$ is an eigenvalue of $d'$. Therefore, $\{i_1',\ldots, i_k'\} \in E_k(G')$ if and only if $\{f'(i_1'),\ldots, f'(i_k')\} \in E_k(G'')$ from which we conclude that $f'$ is a graph isomorphism between $G'$ and $G''$.


\end{proof}

\section{System of Distinct Representatives and Symplectic Forms}

\begin{mydef}
Let $G$ be a hypergraph. A {\it system of distinct representatives} is an injection $f:E(G)\to V(G)$ such that $f(e)\in e$ for every $e\in E(G)$.
\end{mydef}

\begin{example}\label{ex:9}
If $G$ is a (finite simple) graph, then $G$ has a system of distinct representatives if and only if every connected component of $G$ has at most as many edges as vertices. For more general hypergraphs, the condition of having a system of distinct representatives is less trivial. For instance, if $V(G)=\{1,2,3,4,5,6\}$ and $E(G)=\{\{1,2\},\{1,3\},\{2,3\},\{3,4\},\{2,4\},\{4,5,6\}\}$ then $G$ is connected and $|E(G)|=|V(G)|$ but, for instance by applying Hall's Marriage Theorem (see e.g.\ \cite{Jukna11}), $G$ does not have a system of distinct representatives.
\end{example}

\begin{mydef}
A {\it symplectic form} on an $L_\infty$-algebra on the graded vector space $\mathfrak g$ is a linear map $\omega:\mathfrak g\otimes \mathfrak g\to \mathbb R$
that is non-degenerate and closed as an element of $C^\bullet(\mathfrak g)$.
\end{mydef}

\begin{rem}
The following Lemma is a generalization of Proposition 4.8 in \cite{DottiTirao00} to 2-step nilpotent $L_\infty$-algebras.
\end{rem}

\begin{lem}\label{lem:10}
Let $(\mathfrak g,\{l_k\})$ be a 2-step nilpotent $L_\infty$-algebra with symplectic form $\omega$. Then $2\dim l(\mathfrak g)\le \dim \mathfrak g$.
\end{lem}

\begin{proof} Let $x_1,\ldots, x_k,y_1,\ldots,y_j\in \mathfrak g$. Since $\mathfrak g$ is 2-step nilpotent, then
\begin{equation}
l_k(x_1,\ldots,x_{p-1},\widehat{x_p},x_{p+1},\ldots,x_k,l_j(y_1,\ldots,y_j))=0
\end{equation}
for every $p\in \{1,\ldots,k\}$.
Taking into account the definition of the differential of the Maurer-Cartan algebra, and the fact that $\omega$ is closed, we obtain
\begin{equation}
\omega(l_k(x_1,\ldots,x_k),l_j(y_1,\ldots,y_j))=(d_k\omega)(x_1,\ldots,x_k,l_j(y_1,\ldots,y_j))=0\,.
\end{equation}
Hence $l(\mathfrak g)$ is an isotropic subspace for the bilinear form $\omega$ on the real vector space underlying $\mathfrak g$. As a result, the dimension of $l(\mathfrak g)$ is at most half of the dimension of $\mathfrak g$, which is what we needed to prove. 
\end{proof}

\begin{rem}\label{rem:11}
The above calculation shows that a symplectic form $\omega$ on a 2-step nilpotent $L_\infty$-algebra $\mathfrak g$ defines an injective linear map $\omega^\flat:l(\mathfrak g) \to (\mathfrak g/l(\mathfrak g) )^\vee$. In the particular case of $L_\infty$-algebras of the form $\mathcal L(G)$ for some finite simple hypergraph $G$ endowed with a system of distinct representatives $f:E(G)\to V(G)$, we also have an induced map from $l( \mathcal L(G))={\rm span}_\mathbb R(E(G))$ to $(\mathcal L(G)/(l(\mathcal L(G))))^\vee={\rm span}_\mathbb R(V(G))^\vee$ defined by $x_I\mapsto x_{f(I)}^*$ for every $I\in E(G)$. The connection between these two constructions is provided by the following theorem. 
\end{rem}

\begin{theorem}\label{thm:13}
Let $G$ be a finite simple hypergraph. The 2-step nilpotent $L_\infty$-algebra $\mathcal L(G)$ is symplectic if and only if $|E(G)|+|V(G)|$ is even and $G$ has a system of distinct representatives.
\end{theorem}

\begin{proof}
Assume $|E(G)|+|V(G)|$ is even and let $f:E(G)\to V(G)$ be a system of distinct representatives. Then $S=V(G)\setminus f(E(G))$ has an even number of elements. Let $\sigma$  be a fixed-point free involution of $S$. Then
\begin{equation}
    \omega=\sum_{I\in E(G)} x_I^* x_{f(I)}^*+\sum_{i\in S, i<\sigma(I)} x_i^*x_{\sigma(i)}^*
\end{equation}
is a symplectic form on $\mathcal L(G)$. For the converse, assume $\omega$ is a symplectic form on $\mathcal L(G)$. By Lemma \ref{lem:10}, this implies that $|E(G)|\le |V(G)|$. Moreover, $\omega$ is of the form
\begin{equation}
    \omega = \sum_{i,j\in V(G)} \omega_{i,j} x_i^* x_j^*+\sum_{I\in E(G)} \varphi_Ix_I^*\,,
\end{equation}
where $\varphi_I=\omega^\flat(x_I)$. Suppose $G$ does not have a system of distinct representatives. By Hall's Marriage Theorem, for some positive integer $r$ there exist edges $I_1,\ldots,I_r\in E(G)$ such that $|I_1\cup\cdots\cup I_r|<r$. Since, by Remark \ref{rem:11}, ${\rm span}_{\mathbb R}(\varphi_{I_1},\ldots,\varphi_{I_r})$ has dimension $r$, there exist  $p\in \{1,\ldots,r\}$ and $i\notin I_1\cup\cdots\cup I_r$ such that $\varphi_{I_p}(x_i)\neq 0$. Let $I_p=\{i_1,\ldots,i_k\}$. In order for $(d\omega)(x_i,x_{i_1},\ldots,x_{i_k})$ to vanish, there has to be $J\in E_k(G)$ such that $I_p\cup J=\{x_i,x_{i_1},\ldots,x_{i_k}\}$ and $i\in J$. Then $|I_1\cup\cdots\cup I_r\cup J|<r+1$. Iterating this construction we add an edge at each step, until eventually we obtain a collection $\{I_1,\ldots,I_r,\ldots,I_{r'}\}$ of edges such that $|I_1\cup\cdots\cup I_{r'}|<r'$ but $\varphi_{I_p}(x_i)=0$ for all $i\notin I_1\cup\cdots\cup I_{r'}$. This contradicts the linear independence of $\{\varphi_{I_1},\ldots,\varphi_{I_{r'}}\}$ and concludes the proof.
\end{proof}

\begin{rem}
Taking Example \ref{ex:9} into account, Theorem \ref{thm:13} can be thought of as generalizing to hypergraphs the characterization of symplectic Lie algebras associated with graphs found in \cite{PouseeleTirao09}.
\end{rem}

\begin{example}
Let $G$ be a $k$-uniform hypergraph with a transfer filtration of type $|V(G)|-|E(G)|+1$ in the sense of \cite{AdBGS21}. Then $G$ has a system of distinct representatives and thus $\mathcal L(G)$ is symplectic.
\end{example}

\section{Cohomology}

\begin{lem}\label{lem:13}
Let $\mathcal C^\bullet$ be a $\mathbb Z_{\ge 0}$-graded vector space, let $d,\Delta:\mathcal C^\bullet \to \mathcal C^{\bullet+1}$ be commuting differentials, and let $\theta:\mathcal C^\bullet\to \mathcal C^{\bullet-1}$ be a linear map such that $\theta\Delta+\Delta\theta={\rm id}_\mathcal C$. Let $\mathcal D^\bullet = \theta\Delta\mathcal C^\bullet$, and let $D=\theta \Delta d:\mathcal D^\bullet\to \mathcal D^{\bullet+1}$. Then 
\begin{enumerate}[1)]
\item$D^2=0$ and there is a long exact sequence
\begin{equation}\label{eq:27}
    \cdots\to H^{i-1}(\mathcal D^\bullet,D)\xrightarrow{\theta d} H^{i-1}(\mathcal D^\bullet,D)\xrightarrow{\Delta} H^i(\mathcal C^\bullet,d)\xrightarrow{\theta \Delta} H^i(\mathcal D^\bullet,D)\xrightarrow{\theta d} H^i(\mathcal D^\bullet,D)\to \cdots
\end{equation}
\item Assume further that $\Delta=d\varphi-\varphi d$ for some linear map $\varphi:\mathcal C^\bullet\to \mathcal C^\bullet$ such that $[\varphi,\theta]=0$. Then $\ker((\theta d)^n)\subseteq {\rm im}(\theta d)$ for all $n\in \mathbb N$.
\end{enumerate}
\end{lem}

\begin{proof} A lengthy but straightforward verification shows that
\begin{equation}
    0\to (\mathcal D^{\bullet-1}, -D)\xrightarrow{\Delta} (\mathcal C^\bullet, d) \xrightarrow{\theta \Delta} (\mathcal D^\bullet,D)\to 0
\end{equation}
is a short exact sequence of cochain complexes.
Taking the resulting long exact sequence and using the isomorphism $H^{i-1}(\mathcal D^\bullet,D)\cong H^i(\mathcal D^{\bullet-1},-D)$, we obtain \eqref{eq:27}. It is straightforward to check from the definition that the connecting homomorphism is indeed $\theta d$. This proves 1). To prove 2), we first observe that $[\theta d,\varphi]=\theta \Delta$ which acts as the identity on $\mathcal D^\bullet$. If $x\in \mathcal D^\bullet$ is such that $(\theta d)^nx=0$ then
\begin{equation}
    x = \theta d \left(\sum_{k=1}^n
    \frac{\varphi^k (-\theta d)^{k-1}}{k!} \right) x.   
\end{equation}
explicitly showing that $\ker((\theta d)^n)\subseteq {\rm im}(\theta d)$.
\end{proof}

\begin{example}
Let $(\mathcal C^\bullet,d)$ be the Cartan-Chevalley-Eilenberg algebra of a Lie algebra $\mathfrak g$. Let $\alpha\in \mathfrak g^*$ be closed, so that $\ker(\alpha)$ is a codimension 1 ideal of $\mathfrak g$. If $\Delta$ denotes multiplication by $\alpha$ and $\theta$ denotes contraction with an element $x\in \mathfrak g$ such that $\alpha(x)=1$, then $(\mathcal D^\bullet,D)$ is isomorphic to the Cartan-Chevalley-Eilenberg algebra of $\ker(\alpha)$ and \eqref{eq:27} specializes to the Dixmier long exact sequence \cite{Dixmier55}, generalizing \cite{PouseeleTirao09}.
\end{example}

\begin{example}\label{ex:15}
Our main application is to the case in which $(\mathcal C^\bullet,d)$ is the Maurer-Cartan algebra of an $L_\infty$-algebra of the form $\mathcal L(G)$ for some finte simple hypergraph $G$. Generalizing the use of the Dixmier exact sequence in \cite{PouseeleTirao09}, we select a vertex, say, $1\in V(G)$ and remove it from $G$ while keeping track of the edges that contained it. The result is the complex $(\mathcal D^\bullet,D)$ whose underlying graded vector space $\mathcal D^\bullet$ consists of polynomials in ${\rm Sym}((\mathcal L(G))^\vee)$ that do not contain the (odd) variable $x_1^*$. The differential $D$ acts as follows. $D(x_i^*)=0$ for all $i>1$, $D(x_I^*)=0$ if $I$ is an edge of $G$ that contains $1$, and $D(x_{\{i_1,\ldots,i_k\}}^*)=x_{i_1}^*\cdots x_{i_k}^*$ whenever $\{i_1,\ldots,i_k\}\in E_k(G)$ and $i_1<i_2<\cdots<i_k$. If $\Delta$ denotes multiplication by $x_1^*$ and $\theta=\frac{\partial}{\partial x_1^*}$ then it is easy to see that the assumptions of part 1) of Lemma \ref{lem:13} are satisfied. Hence \eqref{eq:27} holds with connecting homomorphism $\delta=\frac{\partial}{\partial x_1^*} \circ d$. In particular, $H^i(\mathcal L(G))$ is generated as a vector space by the coimage of $\delta$ acting on $H^{i-1}(\mathcal D^\bullet,D)$ and the cocycles that map onto the kernel of $\delta$ acting on $H^i(\mathcal D^\bullet,D)$ under the projection $\frac{\partial}{\partial x_1^*} \circ x_1^*$. In the particular case in which $\{1\}\in E_1(G)$, $\Delta=d\varphi-\varphi d$ holds with $\varphi$ being multiplication by $x_{1}^*$. Then the assumptions of part 2) of Lemma \ref{lem:13} are also satisfied. Since $\mathcal D^\bullet$ is generated by monomials of the form $(x_{\{1\}}^*)^rx_{i_1}^*\cdots x_{j_k}^*x_{I_{j_1}}^*\cdots x_{I_{j_l}}^* $,
each of which is $\ker(\delta^N)$ for sufficiently large $N$, then, by Lemma \ref{lem:13}, $\delta$ is surjective. Thus, $H^i(\mathcal L(G))$ is isomorphic, as a vector space, to the kernel of $\delta$ acting on $H^i(\mathcal D^\bullet,D)$.
\end{example}

\begin{theorem}\label{thm:16}
Let $G$ be a finite simple hypergraph. Then $H^i(\mathcal L(G))$ is finite dimensional for every $i\ge 0$.
\end{theorem}

\begin{proof}
The proof is by induction on $|V(G)|$, using Lemma \ref{lem:13} as described in Example \ref{ex:15}. If $\{1\}\notin E_1(G)$ then, by induction hypothesis, $H^i(\mathcal D^\bullet,D)$ and $H^{i-1}(\mathcal D^\bullet,D)$ are finite-dimensional and hence so is $H^i(\mathcal L(G))$ by part 1) of Lemma \ref{lem:13}. On the other hand, if $\{1\}\in E_1(G)$, then part 2) of Lemma \ref{lem:13} shows that $H^i(\mathcal L(G))$ is isomorphic to the kernel of $\delta=\frac{\partial}{\partial x_1^*}\circ d$ acting on $H^i(\mathcal D^\bullet,D)$. The vector space $H^i(\mathcal D^\bullet,D)$ can be further identified with polynomials in $x_{\{1\}}^*$ with coefficients in a finite dimensional vector space $U^i$, which is closed under the action of $\delta$. Moreover, arguing as in Example \ref{ex:15}, $U^i=\ker \delta^N$ for some positive integer $N$. Whence we obtain a filtration
\begin{equation}
    \ker \delta\subseteq \ker \delta^2\subseteq\cdots \subseteq \ker \delta^{N-1}\subseteq \ker \delta ^N=U^i\,.
\end{equation}
Consider a (finite) basis $\{\gamma_j\}$ of $U^i/\ker \delta^{N-1}$. Then the most general element of $H^i(\mathcal D^\bullet,D)$ modulo polynomials in $x_{\{1\}}^*$ with coefficients in $\ker \delta^{N-1}$ is of the form $\sum_j Q_j \gamma_j$ for some $Q_j\in \mathbb R[x_{\{1\}}^*]$. Hence $\delta(\sum_j Q_j\gamma_j)=\sum_j Q'_j\gamma_j$ modulo polynomials in $x_{\{1\}}^*$ with coefficients in $\ker \delta^{N-1}$, where $Q_j'$ denotes the derivative of $Q_j$ with respect to $x_{\{1\}}^*$. Therefore, $\sum_j Q_j \gamma_j$ is in the kernel of $\delta$ if and only if $Q_j$ is a constant polynomial for all $j$. Similarly, let $\{\beta_k\}$ be a basis of $\ker \delta^{N-1}/\ker \delta^{N-2}$ so that the most general element in $H^i(\mathcal D^\bullet,D)/(\ker \delta^{N-2})[x_{\{1\}}^*]$ can be written as $\sum_j Q_j \gamma_j+\sum_k R_k \beta_k$, $R_k\in \mathbb R[x_{\{1\}}^*]$.  Moreover, there exist real numbers $a_{jk}$ such that $\delta \gamma_j=\sum_k a_{jk}\beta_k$ for all $j$. Hence
\begin{equation}
    \delta\left(\sum_j Q_j \gamma_j + \sum_k R_k \beta_k\right) = \sum_k\left(\sum_j a_{jk}Q_k+R_k'\right)\beta_k   
\end{equation}
modulo polynomials with coefficients in $\ker \delta^{N-2}$. It follows that $\sum_j a_{jk}Q_k+ R_k'=0 $ so that, for each $k$, $R_k$ is a polynomial of degree at most one in the variable $x_{\{1\}}^*$. Continuing the process, we see that in order for a linear combination in $H^i(\mathcal D^\bullet,D)/(\ker^{N-p})$ to lift to an element in the kernel of $\delta$ acting on $H^i(\mathcal D,D)$, the degree of the polynomials in $x_{\{1\}}^*$ appearing as coefficients cannot be greater than $p-1\le N$.

\end{proof}

\begin{mydef}\label{def:26}
Let $G$ be a finite simple hypergraph. For every $i\ge0$, the $i$-th {\it $L_\infty$-Betti number of $G$} is the integer $b_i(G)=\dim H^i(\mathcal L(G))$. The {\it $L_\infty$-Poincar\'e series of $G$} is the formal power series $P_t(G)=\sum_{i=0}^\infty b_i(G)t^i$.
\end{mydef}

\begin{rem}
The terminology of Definition \ref{def:26} is justified by the fact that if $G$ is a graph then, by a theorem of Nomizu \cite{Nomizu54}, $P_t(G)$ coincides with the (topological) Poincar\'e 
 polynomial of the compact nilmanifold associated to $\mathcal L(G)$.
\end{rem}

\begin{rem}
Suppose $E_1(G)=\emptyset$. Then Theorem \ref{thm:16} is trivial since $C^k(\mathcal L(G))$ is finite dimensional for every $k$. More precisely, the generating function for $\dim ( C^k(\mathcal L(G))$ is 
\begin{equation}\label{eq:32}
  (1+t)^{|V(G)|}  \frac{\prod_{p=1}^\infty (1+t^{2p-1})^{|E_{2p}(G)|}}{\prod_{q=1}^\infty(1-t^{2q})^{|E_{2q+1}(G)|}}\,,
\end{equation}
which provides explicit upper bounds for the $L_\infty$-Betti numbers of $G$. In particular, if $G$ is a hypergraph such that $E_{2k+1}(G)=\emptyset$ for all $k\in \mathbb N$ then $P_t(G)$ is a polynomial. 
\end{rem}

\begin{example}
Let $G$ be the hypergraph with $k$ vertices and a single $k$-edge. If $k$ is even, then \eqref{eq:32} specializes to $(1+t)^k(1+t^{k-1})$. Moreover, in this case the differential maps $x_{\{1,2,\ldots,k\}}^*$ to $x_1^*\cdots x_k^*$ and acts trivially on other monomials. Hence
\begin{equation}
 P_t(G)=(1+t)^k(1+t^{k-1})-t^{k-1}-t^k\,.
\end{equation}
On the other hand, if $k$ is odd then \eqref{eq:32} specializes to 
\begin{equation}
\frac{(1+t)^k}{1-t^{k-1}}
\end{equation}
and the only non-zero differentials are the ones mapping 
\begin{equation}
(x_{\{1,2,\ldots,k\}}^*)^j\mapsto  jx_1^*\cdots x_k^*(x_{\{1,2,\ldots,k\}}^*)^{j-1}
\end{equation}
for every $j\in \mathbb N$. Hence
\begin{equation}
P_t(G)=\frac{(1+t)^k-t^{k-1}-t^k}{1-t^{k-1}}\,.
\end{equation}
\end{example}

\begin{rem}
Let $G_1$, $G_2$ be finite simple hypergraphs and consider their disjoint union $G_1+G_2$. Then $\mathcal L(G_1+G_2)=\mathcal L(G_1)\oplus \mathcal L(G_2)$ and, using K\"unneth's Theorem, $P_t(G_1+G_2)=P_t(G_1)P_t(G_2)$. 
\end{rem}

\begin{rem}
In all examples that we have calculated, the $L_\infty$-Poincar\'e series $P_t(G)$ is a rational function of $t$. It would be interesting to know if this is a general feature. 
\end{rem}

\begin{example}
Let $G_1$ be the hypergraph with $|V(G)|=1=|E_1(G)|$. A straightforward calculation shows that $P_t(G_1)=1$. Hence, for any hypergraph $G_2$, $P_t(G_1+G_2)=P_t(G_1)P_t(G_2)=P_t(G_2)$. In particular, we see that the $L_\infty$-Poincar\'e  series does not distinguish the non-isomorphic hypergraphs $G_1+G_2$ and $G_2$ and hence is not a complete invariant of hypergraphs. 
\end{example}

\begin{prop}\label{prop:20}
Let $G$ be a finite simple hypergraph. Then $b_0(G)=1$.
\end{prop}

\begin{proof}
It follows from \eqref{eq:27} that $H^0(\mathcal L(G))$ is isomorphic to the kernel of the connecting homomorphism $\delta$ acting on $H^0(\mathcal D^\bullet,D)$. Since $\mathcal D^0$ is the space of real polynomials in the variables $x_{\{i\}}^*$ whenever $\{i\}\in E_1(G)$, keeping in mind the action of $D$ as described in Example \ref{ex:15}, we conclude that $H^0(\mathcal D^\bullet,D)$ is isomorphic to $\mathbb R$ if $\{1\}\notin E_1(G)$ and $\delta$ acts as zero. Similarly, $\mathbb R[x_{\{1\}}^*]$ if $\{1\}\in E_1(G)$ and $\delta(P(x_{\{1\}}^*))=P'(x_{\{1\}}^*)$ for all $P\in \mathbb R[x_{\{1\}}^*]$.
\end{proof}

\begin{prop}\label{prop:25}
Let $G$ be a finite simple hypergraph and let $N$ be the smallest integer such that $E_N(G)\neq \emptyset$. Then $b_i(G)=\binom{|V(G)|}{i}$ for all $i<N$.
\end{prop}

\begin{proof}
The case $N=1$ is the statement of Proposition \ref{prop:20}.  If $1\le i<N$, then the result follows from the fact that $C^i(\mathcal L(G))$ consists of degree $i$ polynomials in the odd variables $x_j^*$ on which $d$ acts trivially. 
\end{proof}

\begin{example}
If $N=2$, then $b_1(G)=|V(G)|$ (as proved in \cite{PouseeleTirao09} for graphs).
\end{example} 

\begin{prop}
Let $G$ be a finite simple hypergraph. Then $b_1(G)-|V(G)|+|E_1(G)|$ is the number of 2-edges $\{i,j\}$ such that $\{k\}\in E_1(G)$ for at least one $k\in \{i,j\}$.
\end{prop}

\begin{proof}
If $E_1(G)=\emptyset$ then the statement reduces to $b_1(G)=|V(G)|$, which is a particular case of Proposition \ref{prop:25}. Hence we can work by induction on the number of $1$-edges. Assume $\{1\}\in E_1(G)$ and let $G\setminus \{1\}$ be the hypergraph obtained by removing $1\in V(G)$ and all the edges $e\in E(G)$ such that $1\in e$. By induction, it suffices to show that if $\{1\}\in E_1(G)$, then $b_1(G)-b_1(G\setminus\{1\})$ is the number of $e\in E_2(G)$ such that $1\in e$. This can be shown using Lemma \ref{lem:13}, as follows. Since $\{1\}\in E_1(G)$, then $H^1(\mathcal L(G))$ is isomorphic to the kernel of $\delta$ acting on $H^1(\mathcal D^\bullet,D)$. The most general element of $H^1(\mathcal D^\bullet,D)$ is of the form
\begin{equation}
\xi=\sum_\gamma P_\gamma(x_{\{1\}}^*)\gamma + \sum_j Q_j(x_{\{1\}}^*)x_{\{1,j\}}^*\,,
\end{equation}
where $\gamma$ ranges over a  basis of $H^1(\mathcal L(G\setminus \{1\}))$, $j$ ranges over all the vertices connected to $1$ by a 2-edge, and $P_i,Q_j\in \mathbb R[x_{\{1\}}^*]$. Imposing $\delta(\xi)=0$ forces the $Q_j$ to be constants. If $\{j\}\in E_1(G)$, then $\delta(x_{\{1,j\}}^*)=x_j^*=d(x_{\{j\}}^*)$ vanishes in $H^1(\mathcal D^\bullet,D)$. If $\{j\}\notin E_1(G)$, then $\delta(\xi)=0$ forces $Q_j=P'_{x_j^*}$. On the other hand, if $\gamma\in H^1(\mathcal L(G\setminus\{1\}))$ is not of the form $x_j^*$ with $\{1,j\}\in E_2(G)$ then $P_\gamma'=0$. 
Hence the kernel of the action of $\delta$ on $H^1(\mathcal D^\bullet,D)$ is isomorphic to the direct sum of $H^1(\mathcal L(G\setminus\{1\}))$ with the span of all the 2-edges that contain $1$ as a vertex. This proves that $b_1(G)-b_1(G\setminus\{1\})$ is the number of $e\in E_2(G)$ such that $1\in e$ and thus the result.
\end{proof}

\begin{prop}
Let $G$ be a finite simple hypergraph and let $N$ be the smallest integer such that $E_N(G)\neq \emptyset$. If $N>1$, then
\begin{equation}\label{eq:38}
    b_N(G)=\binom{|V(G)|}{N}-\binom{|V(G)|}{N+1}+(N-1)|E_N(G)| +\sum_{S\subseteq \binom{V(G)}{N+1}}\min(\nu(S),1)\,,
\end{equation}
where $\nu(S)$ is the number $N$-edges with vertices all contained in $S$. 
\end{prop}

\begin{proof}
$C^N(\mathcal L(G))$ is generated by two types of monomials: type I) $x_{i_1}^*\cdots x_{i_N}^*$, and type II) $x_{i_0}^*x_{\{i_1,\ldots,i_N\}}^*$. There are $\binom{|V(G)|}{N}$ monomials of type I) and they are all cocycles. Of these $|E_N(G)|$ are coboundaries for a total contribution of $\binom{|V(G)|}{N}-|E_N(G)|$ linearly independent cocycles from monomials of type I). Since $N$ is the smallest possible edge size, no linear combination of monomials of type II) can be a coboundary. In order to count type II) cocycles, it is helpful to consider two subcases. If $i_0\in \{i_1,\ldots,i_N\}$, then $x_{i_0}^*x_{\{i_1,\ldots,i_N\}}^*$ is a cocycle and there are $N|E_N(G)|$ of those. If $i_0\notin \{i_1,\ldots,i_N\}$, we may assume without loss of generality that $i_0<i_1<\cdots<i_N$. Then $x_{i_0}^*x_{\{i_1,\ldots,i_N\}}^*$ is not a cocycle. However, consider linear combinations of the form
\begin{equation}
\xi= \sum \alpha_j x_{i_j}^*x_{\{i_0,i_i,\ldots,\widehat{i_j},i_N\}}^*\,,
\end{equation}
where the sum ranges over all $j\in \{0,\ldots,N\}$ such that $\{i_0,i_i,\ldots,\hat{i_j},i_N\}$ is in $E_N(G)$. Then $\xi$ is a cocycle if and only if $\sum (-1)^j\alpha_j=0$. Due to this linear condition, the number of linearly independent cocycles that the subset $\{i_0,i_1,\ldots,i_N\}\in \binom{V(G)}{N+1}$ contributes to the $N$-th cohomology group is $\nu(S)-1$. Adding all of these contributions together easily yields \eqref{eq:38}.
\end{proof}

\begin{example}\label{ex:29}
Let $G$ be the complete $k$-uniform hypergraph on $n$ vertices. Then 
\begin{equation}
b_k(G)=\binom{n}{k}-\binom{n}{k+1}+(k-1)\binom{n}{k}+(k+1)\binom{n}{k+1}=k\binom{n+1}{k+1}\,.
\end{equation}
\end{example}

\begin{example}
If $N=2$, then
\begin{equation}
\binom{|V(G)|}{3}-\sum_{S\subseteq \binom{V(G)} {3}}\min(\nu(S),1) 
\end{equation}
is equal to the difference between number of triangles and the number of pairs of adjacent 2-edges. Hence, specializing \eqref{eq:38} to the case of graphs, we recover the formula for $b_2(G)$ proved in \cite{PouseeleTirao09}.
\end{example}

\section*{Acknowledgements} We would like to thank the anonymous referee for valuable feedback. This work is supported in part by VCU Quest Award ``Quantum Fields and Knots: An Integrative Approach.''

\vskip.1in\noindent
\address{Marco Aldi\\
Department of Mathematics and Applied Mathematics\\
Virginia Commonwealth University\\
Richmond, VA 23284, USA\\
\email{maldi2@vcu.edu}}

\vskip.1in\noindent
\address{Samuel Bevins\\
Department of Mathematics and Applied Mathematics\\
Virginia Commonwealth University\\
Richmond, VA 23284, USA\\
\email{bevinssj@vcu.edu}}

\end{document}